\documentclass[12pt, reqno]{amsart}
\usepackage{amsmath, amsthm, amscd, amsfonts, amssymb, graphicx, color}
\input{mathrsfs.sty}

\newtheorem{theorem}{Theorem}[section]
\newtheorem{lemma}[theorem]{Lemma}

\newtheorem{corollary}[theorem]{Corollary}
\theoremstyle{definition}

\newtheorem{example}[theorem]{Example}

\theoremstyle{remark}
\newtheorem{remark}[theorem]{Remark}

\begin{document}

\title[]{   A refinement of  Hardy Inequality via superquadratic function}
\author[Mohsen Kian, M. Rostamian Delavar]{  Mohsen Kian and M. Rostamian Delavar }

\address{${^{1,2}}$Department of Mathematics, Faculty of Basic Sciences, University of Bojnord, P. O. Box
1339, Bojnord 94531, Iran}
\email{\textcolor[rgb]{0.00,0.00,0.84}{kian@ub.ac.ir \ , \ kian@member.ams.org }} 
\email{\textcolor[rgb]{0.00,0.00,0.84}{m.rostamian@ub.ac.ir}}



\begin{abstract}
A refinement of the Hardy inequality has been presented by use of superquadratic function.
\end{abstract}

\maketitle
\section{Introduction}

The classical Hardy inequality asserts that if $p>1$ and $f$ is a non-negative $p$-integrable function on $(0,\infty)$, then
 \begin{eqnarray}\label{n1}
 \int_{0}^{\infty}\left(\frac{1}{x}\int_{0}^{x}f(t)dt\right)^pdx\leq
  \left(\frac{p}{p-1}\right)^p\int_{0}^{\infty}f(x)^pdx.
 \end{eqnarray}

 Let $p>1$ and let $F:(0,\infty)\to\mathcal{B}(\mathcal{H})^+$  be a weakly measurable mapping such that
  $$\int_{0}^{\infty}F(x)^pdx\in\mathcal{B}(\mathcal{H}).$$
Assume that the real valued function $f$ is defined on $(0,\infty)$ by $f(t)=\langle F(t)\eta,\eta\rangle$. Then $f$ is $p$-integrable, since
\begin{align*}
  \int_{0}^{\infty}f(t)^pdt &= \int_{0}^{\infty}\langle F(t)\eta,\eta\rangle^pdt\\
  &\leq \int_{0}^{\infty}\langle F(t)^p\eta,\eta\rangle dt\qquad \mbox{by \eqref{MP}}\\
  &=\left\langle \int_{0}^{\infty}  F(t)^pdt\ \eta,\eta \right\rangle.
\end{align*}
The classical Hardy inequality \eqref{n1} now implies that
\begin{align}\label{q1}
  \int_{0}^{\infty}\left(\frac{1}{x}\int_{0}^{x}\langle F(t)\eta,\eta\rangle \ dt\right)^pdx\leq\left(\frac{p}{p-1}\right)^p
  \int_{0}^{\infty}\langle F(t)\eta,\eta\rangle^p\ dt.
\end{align}
In the case where $p\in(1,2]$, Hansen \cite{H} proved that  a stronger form of \eqref{q1}  holds true:
\begin{align}\label{hansen}
  \int_{0}^{\infty}\left(\frac{1}{x}\int_{0}^{x}F(t)dt\right)^pdx\leq\left(\frac{p}{p-1}\right)^p
  \int_{0}^{\infty}F(t)^pdt.
\end{align}
 However, if $p>2$, the inequality \eqref{hansen} is not valid in general, see \cite{H}. In this paper, utilizing the notion of the superquadratic functions, we give an improvement of the Hardy inequality \eqref{n1} for $p\geq2$. Furthermore, our result will provide some difference counterpart to Hardy inequality.

\section{Preliminaries}
Superquadratic functions have been introduced as a modification of convex functions in \cite{ajs}. A function $f:[0,\infty)\to\mathbb{R}$ is said to be superquadratic  whenever for all $a\geq0$ there exists a constant $C_a\in\mathbb{R}$  such that
  \begin{eqnarray}\label{def}
    f(b)\geq f(a)+C_a(b-a)+f(|b-a|)
  \end{eqnarray}
for all $b\geq0$. If such  $f$ is positive, then it is convex too and a sharper Jensen inequality holds true: For every probability measure $\mu$ on $\Omega$ and every $\mu$-integrable function $\varphi$ on $\Omega$, if $f$ is superquadratic, then
\begin{align}\label{p1}
f\left( \int_\Omega \varphi(t) d\mu(t)\right)\leq \int_\Omega f(\varphi(t)) d\mu(t)-\int_\Omega f\left(\left|\varphi(t)-\int_\Omega \varphi(s)d\mu(s)\right|\right)d\mu(t).
\end{align}

Assume that $\mathcal{B}(\mathcal{H})$ is the $C^*$-algebra of all bounded linear operators on a Hilbert space $\mathcal{H}$ and $I$ is the identity operator. An operator extension of \eqref{p1} has been presented in \cite{kian}: \\
\textbf{Theorem A.} If  $f:[0,\infty)\to\mathbb{R}$  is a continuous superquadratic  function, then
     \begin{align}\label{kian}
  f(\langle A\eta,\eta\rangle)\leq \langle f(A)\eta,\eta\rangle-\langle f(|A-\langle A\eta,\eta\rangle|)\eta,\eta\rangle
\end{align}
 for every positive operator $A$ and every unit vector $\eta  \in\mathcal{H}$.

 Theorem A provides a refinement of the well-known inequality (see \cite{mp})
 \begin{align}\label{MP}
   g(\langle A \eta,\eta\rangle)\leq \langle g(A)\eta,\eta\rangle,
 \end{align}
   which holds for every continuous convex function $g$. Moreover, a generalization of \eqref{kian}  has been shown in \cite{kd}:\\
\textbf{Theorem B.} Let $\Phi$ be a unital positive linear mapping on $\mathcal{B}(\mathcal{H})$.  If $f:[0,\infty)\to\mathbb{R}$  is a continuous superquadratic  function, then
  \begin{align}\label{main2}
     f(\langle \Phi(A)\eta,\eta\rangle)&\leq \langle \Phi(f(A))\eta,\eta\rangle
   -\left\langle \Phi\left(f\left(\left|A-\langle \Phi(A)\eta,\eta\rangle\right|\right)\right)\eta,\eta\right\rangle.
  \end{align}
for every  positive   operator $A$ and every unit vector $\eta\in\mathcal{H}$.

\section{Refinement of Hardy inequality}
\begin{lemma}\label{lm1}
 Let $p\geq 2$.   If the mappings $G:(0,\infty)\to\mathcal{B}(\mathcal{H})^+$ is   weakly measurable  such that
    $$\int_{0}^{\infty}G(x)^p\frac{dx}{x}\in\mathcal{B}(\mathcal{H}),$$
    then
{\small\begin{align*}
 \int_{0}^{\infty} \left( \frac{1}{x}\int_{0}^{x}\langle G(t)   \eta,\eta \rangle dt\right)^p \ \frac{dx}{x} &\leq
  \int_{0}^{\infty}\left\langle G(t)^p\eta,\eta \right\rangle\frac{dt}{t} \\
  &\quad-
 \int_{0}^{\infty} \frac{1}{x}\int_{0}^{x}\left\langle \left|G(t)  -\left\langle\frac{1}{x}\int_{0}^{x}G(s)ds\ \eta,\eta\right\rangle \right|^p \eta,\eta\right\rangle dt  \frac{dx}{x}.
 \end{align*}}
\end{lemma}
\begin{proof}
Let $\mathscr{A}$ be the $C^*$-algebra of all weakly measurable mappings $F:[0,x]\to\mathcal{B}(\mathcal{H})$ with  $\int_{0}^{x}F(t)dt\in\mathcal{B}(\mathcal{H})$.  Define the unital positive linear mapping $\Phi:\mathscr{A}\to\mathcal{B}(\mathcal{H})$ by
$$\Phi(F)=\frac{1}{x}\int_{0}^{x}F(t)dt,\qquad \forall F\in\mathscr{A}.$$
If $p\geq2$, then the function $f:(0,\infty)\to(0,\infty)$   defined by  $f(t)=t^p$ is superquadratic. If  $\eta\in\mathcal{H}$ is a unit vector, then Theorem B implies that
{\small\begin{align*}
  \left\langle \frac{1}{x}\int_{0}^{x}G(t)dt\ \eta,\eta\right\rangle^p&\leq
  \left\langle \frac{1}{x}\int_{0}^{x}G(t)^pdt\ \eta,\eta\right\rangle\\
  &\quad-
  \left\langle \frac{1}{x}\int_{0}^{x}\left|G(t)  -\left\langle\frac{1}{x}\int_{0}^{x}G(s)ds\ \eta,\eta\right\rangle \right|^pdt\ \eta,\eta\right\rangle
\end{align*}}
for every $x>0$. Multiplying both sides  by $\frac{1}{x}$ and integrating over $(0,\infty)$ we obtain
{\small\begin{align}\label{q2}
 \int_{0}^{\infty} \left\langle \frac{1}{x}\int_{0}^{x}G(t)dt\ \eta,\eta\right\rangle^p \ \frac{dx}{x} &\leq
  \int_{0}^{\infty}\left\langle \frac{1}{x}\int_{0}^{x}G(t)^pdt\ \eta,\eta\right\rangle \frac{dx}{x}\nonumber\\
  &\qquad-
 \int_{0}^{\infty} \left\langle \frac{1}{x}\int_{0}^{x}\left|G(t)  -\left\langle\frac{1}{x}\int_{0}^{x}G(s)ds\ \eta,\eta\right\rangle \right|^pdt\ \eta,\eta\right\rangle\ \frac{dx}{x}.
\end{align}}
 Noting  that
 {\small\begin{align*}
  \int_{0}^{\infty}\left\langle \frac{1}{x}\int_{0}^{x}G(t)^pdt\ \eta,\eta\right\rangle \frac{dx}{x} &=\int_{0}^{\infty} \frac{1}{x}\int_{0}^{x}\left\langle G(t)^p\eta,\eta \right\rangle dt\frac{dx}{x}\\
  &= \int_{0}^{\infty}\left\langle G(t)^p\eta,\eta \right\rangle \int_{t}^{\infty}\frac{1}{x^2}dx\ dt\\
  &= \int_{0}^{\infty}\left\langle G(t)^p\eta,\eta \right\rangle\frac{dt}{t},
\end{align*}}
 we get from  \eqref{q2}   that
 {\small\begin{align}\label{q3}
 \int_{0}^{\infty} \left\langle \frac{1}{x}\int_{0}^{x}G(t)dt\ \eta,\eta\right\rangle^p \ \frac{dx}{x} &\leq
  \int_{0}^{\infty}\left\langle G(t)^p\eta,\eta \right\rangle\frac{dt}{t}\nonumber\\
  &\quad-
 \int_{0}^{\infty} \frac{1}{x}\int_{0}^{x}\left\langle \left|G(t)  -\left\langle\frac{1}{x}\int_{0}^{x}G(s)ds\ \eta,\eta\right\rangle \right|^p \eta,\eta\right\rangle dt  \frac{dx}{x}.
\end{align}}

 \end{proof}
\begin{theorem}\label{th2}
 Let $  p\geq 2$. If the mapping $F :(0,\infty)\to\mathcal{B}(\mathcal{H})^+$ is  weakly measurable  such that
    $$\int_{0}^{\infty}F(x)^pdx\in\mathcal{B}(\mathcal{H}),$$
    then
{\small\begin{align*}
 \int_{0}^{\infty} &\left\langle\frac{1}{x}\int_{0}^{ x}F(t)dt \ \eta,\eta\right\rangle^p dx\\
  &\leq
 \left(\frac{p}{p-1} \right)^p\int_{0}^{\infty}\left\langle F(x)^p\ \eta,\eta\right\rangle\ dx \\
    & \ \ -\left(\frac{p}{p-1} \right)^{p-2}\left\langle   \int_{0}^{\infty} \frac{1}{x}\int_{0}^{x} x^{\frac{1}{p}}t^{\frac{-1}{p}} \left|x^{\frac{-1}{p}}t^{\frac{1}{p}}F(t) -\frac{p-1}{p}\left\langle\frac{1}{x}\int_{0}^{x}F(r)dr\ \eta,\eta\right\rangle \right|^p dt\  dx \ \eta,\eta \right\rangle.
\end{align*}}
for every unit vector $\eta\in\mathcal{H}$.
\end{theorem}

\begin{proof}
 We use  Lemma \ref{lm1} and proceed as  argument applied in \cite[Theorem 2.3]{H}. Put
   $G(t)=F(t^{\frac{p}{p-1}})t^{\frac{1}{p-1}}$ so that    $G$ is weakly measurable. Applying Lemma  \ref{lm1} to $G$ we get
{\small\begin{align}\label{q4}
 \int_{0}^{\infty}& \left\langle \frac{1}{x}\int_{0}^{x}F(t^{\frac{p}{p-1}})t^{\frac{1}{p-1}}dt\ \eta,\eta\right\rangle^p \ \frac{dx}{x} \nonumber\\
 &\leq
  \int_{0}^{\infty}\left\langle F(t^{\frac{p}{p-1}})^pt^{\frac{p}{p-1}}\eta,\eta \right\rangle\frac{dt}{t}\nonumber\\
  &\quad-
 \int_{0}^{\infty} \frac{1}{x}\int_{0}^{x}\left\langle \left|F(t^{\frac{p}{p-1}})t^{\frac{1}{p-1}} -\left\langle\frac{1}{x}\int_{0}^{x}F(s^{\frac{p}{p-1}})s^{\frac{1}{p-1}}ds\ \eta,\eta\right\rangle \right|^p \eta,\eta\right\rangle dt  \frac{dx}{x}.
\end{align}}
Let we use the symbol $I\leq II-III$ for \eqref{q4}.  With substituting $y=t^{\frac{p}{p-1}}$ and $dy=\frac{p}{p-1}t^{\frac{1}{p-1}}dt$  we obtain
{\small\begin{align}\label{q12}
 I=\int_{0}^{\infty}& \left\langle \frac{1}{x}\int_{0}^{x}F(t^{\frac{p}{p-1}})t^{\frac{1}{p-1}}dt\ \eta,\eta\right\rangle^p \ \frac{dx}{x}
=\left(\frac{p-1}{p}\right)^p\int_{0}^{\infty}\left\langle\frac{1}{x}
              \int_{0}^{ x^{\frac{p}{p-1}}}F(y)dy \ \eta,\eta\right\rangle^p\frac{dx}{x}
\end{align}}
and
{\small\begin{align*}
 II= \int_{0}^{\infty}\left\langle F(t^{\frac{p}{p-1}})^pt^{\frac{p}{p-1}}\eta,\eta \right\rangle\frac{dt}{t}=\int_{0}^{\infty}\left\langle F(y)^p\ \eta,\eta\right\rangle\ dt.
\end{align*}}
Moreover, using substituting $r=s^{\frac{p}{p-1}}$ and $dr=\frac{p}{p-1}s^{\frac{1}{p-1}}ds$  we get {\small\begin{align}\label{q333}
III=\int_{0}^{\infty}& \frac{1}{x}\int_{0}^{x}\left\langle \left|F(t^{\frac{p}{p-1}})t^{\frac{1}{p-1}} -\left\langle\frac{1}{x}\int_{0}^{x}F(s^{\frac{p}{p-1}})s^{\frac{1}{p-1}}ds\ \eta,\eta\right\rangle \right|^p \eta,\eta\right\rangle dt  \frac{dx}{x}\nonumber\\
&=\frac{p-1}{p}\int_{0}^{\infty} \frac{1}{x}\int_{0}^{x^{\frac{p}{p-1}}}\left\langle \left|F(y)y^{\frac{1}{p}} -\frac{p-1}{p}\left\langle\frac{1}{x}\int_{0}^{x^{\frac{p}{p-1}}}F(r)dr\ \eta,\eta\right\rangle \right|^p \eta,\eta\right\rangle y^{\frac{-1}{p} }dy  \frac{dx}{x}.
\end{align}}
Applying the substitution $z=x^{\frac{p}{p-1}}$ and $\frac{dz}{z}=\frac{p}{p-1}\frac{dx}{x}$ in \eqref{q12} and \eqref{q333} respectively, we can write
{\small\begin{align}\label{q121}
 I=\left(\frac{p-1}{p}\right)^p\int_{0}^{\infty}\left\langle\frac{1}{x}
              \int_{0}^{ x^{\frac{p}{p-1}}}F(y)dy \ \eta,\eta\right\rangle^p\frac{dx}{x}=\left(\frac{p-1}{p}\right)^p\int_{0}^{\infty}
              \left\langle\frac{1}{z}
              \int_{0}^{ z}F(y)dy \ \eta,\eta\right\rangle^p dz,
\end{align}}
and
{\small\begin{align*}
III&=\int_{0}^{\infty} \frac{1}{x}\int_{0}^{x^{\frac{p}{p-1}}}\left\langle \left|F(y)y^{\frac{1}{p}} -\frac{p-1}{p}\left\langle\frac{1}{x}\int_{0}^{x^{\frac{p}{p-1}}}F(r)dr\ \eta,\eta\right\rangle \right|^p \eta,\eta\right\rangle y^{\frac{-1}{p} }dy  \frac{dx}{x}\\
&= \left(\frac{p-1}{p}\right)^2\int_{0}^{\infty} \frac{1}{z}\int_{0}^{z}\left\langle  z^{\frac{1}{p}}y^{\frac{-1}{p}} \left|z^{\frac{-1}{p}}y^{\frac{1}{p}}F(y) -\frac{p-1}{p}\left\langle\frac{1}{z}\int_{0}^{z}F(r)dr\ \eta,\eta\right\rangle \right|^p \eta,\eta\right\rangle  dy\  dz\\
&=\left(\frac{p-1}{p}\right)^2\left\langle   \int_{0}^{\infty} \frac{1}{z}\int_{0}^{z} z^{\frac{1}{p}}y^{\frac{-1}{p}} \left|z^{\frac{-1}{p}}y^{\frac{1}{p}}F(y) -\frac{p-1}{p}\left\langle\frac{1}{z}\int_{0}^{z}F(r)dr\ \eta,\eta\right\rangle \right|^p dy\  dz \ \eta,\eta \right\rangle.
\end{align*}}

\end{proof}
Assume that $f$ is a positive $p$-integrable $(p\geq2)$ function on $(0,\infty)$. Consider the mapping $F:(0,\infty)\to \mathcal{B}(\mathcal{H})^+$ defined by $F(t)=f(t) I$. Theorem \ref{th2} then gives the following refinement of the classical Hardy inequality \eqref{n1}.
\begin{corollary}\label{co1}
  If  $p\geq2$ and $f:(0,\infty)\to(0,\infty)$ is a $p$-integrable function, then
  {\small\begin{align*}
 \int_{0}^{\infty}& \left( \frac{1}{x}\int_{0}^{ x}f(t)dt \right)^p dx\\
  &\leq
 \left(\frac{p}{p-1} \right)^p\int_{0}^{\infty} f(x)^p  dx \\
    & \ \ -\left(\frac{p}{p-1} \right)^{p-2}   \int_{0}^{\infty} \frac{1}{x}\int_{0}^{x} x^{\frac{1}{p}}t^{\frac{-1}{p}} \left|x^{\frac{-1}{p}}t^{\frac{1}{p}}f(t) -\frac{p-1}{p} \frac{1}{x}\int_{0}^{x}f(r)dr  \right|^p dt\  dx.
\end{align*}}
\end{corollary}
We give an example to show that Corollary \ref{co1} really gives an improvement of the classical Hardy inequality \eqref{n1}. The calculations in the next example has been done by the Mathematica software.
\begin{example}
Put $p=2$ and assume that $f(t)=\frac{1}{t+1}$ so that $f$ is square-integrable. Then
 $$\int_{0}^{\infty} \left( \frac{1}{x}\int_{0}^{ x}\frac{1}{t+1}dt \right)^2=\frac{\pi^2}{3},\qquad\int_{0}^{\infty}\frac{1}{(x+1)^2} dx=1,$$
 $$\int_{0}^{\infty} \frac{1}{x}\int_{0}^{x} x^{\frac{1}{2}}t^{\frac{-1}{2}} \left|\frac{x^{\frac{-1}{2}}t^{\frac{1}{2}}}{t+1} - \frac{1}{2x}\int_{0}^{x}\frac{1}{r+1}dr  \right|^2 dt\  dx=2-\frac{\pi^2}{6},$$
and Corollary \ref{co1} gives
$$\frac{\pi^2}{3}\leq 4-(2-\frac{\pi^2}{6}),$$
while the Hardy inequality \eqref{n1} gives $\frac{\pi^2}{3}\leq 4$.
\end{example}
\begin{remark}
   It will be helpful to point out that the power function $f(t)=-t^p$ is superquadratic for $1<p\leq 2$. A same argument as in Theorem \ref{th2} will provide a difference counterpart to the Hardy inequality. With assumption as in Theorem \ref{th2} except $1<p\leq 2$, we obtain
   \begin{align*}
  &\left(\frac{p}{p-1} \right)^p\int_{0}^{\infty}\left\langle F(x)^p\ \eta,\eta\right\rangle\ dx\\
  &\leq
  \int_{0}^{\infty} \left\langle\frac{1}{x}\int_{0}^{ x}F(t)dt \ \eta,\eta\right\rangle^p dx
  \\
    & \ \ +\left(\frac{p}{p-1} \right)^{p-2}\left\langle   \int_{0}^{\infty} \frac{1}{x}\int_{0}^{x} x^{\frac{1}{p}}t^{\frac{-1}{p}} \left|x^{\frac{-1}{p}}t^{\frac{1}{p}}F(t) -\frac{p-1}{p}\left\langle\frac{1}{x}\int_{0}^{x}F(r)dr\ \eta,\eta\right\rangle \right|^p dt\  dx \ \eta,\eta \right\rangle.
   \end{align*}
   \end{remark}

\section{External Jensen inequality for superquadratic functions}

\begin{theorem}
  Let $f:[0,\infty)\to\mathbb{R}$ be a superquadratic function. Let $x,y\in\mathcal{H}$ with $\|x\|^2-\|y\|^2=1$. If $A,B$ are two positive operators, then
{\small  \begin{align*}
   f(\langle Ax,x\rangle -\langle By,y\rangle)&\geq \|x\|^2f\left(\left\langle A \frac{x}{\|x\|},\frac{x}{\|x\|}\right\rangle\right)-\langle f(B)y,y\rangle+\langle f\left(\left|B-\frac{1}{\|y\|^2}\langle  By,y\rangle\right|\right)y,y\rangle\\
 &\ \ +f\left(\|y\|^2 \left|\frac{1}{ \|x\|^2}\langle Ax,x\rangle -\frac{1}{ \|y\|^2}\langle By,y\rangle \right|\right)\\
 & \ \ + \|y\|^2 f\left( \left|\frac{1}{ \|x\|^2}\langle Ax,x\rangle -\frac{1}{ \|y\|^2}\langle By,y\rangle \right|\right),
  \end{align*}}
  provided that
  $\langle Ax,x\rangle -\langle By,y\rangle\geq0$.
\end{theorem}

\begin{align}\label{qq1}
  f\left(\left\langle A \frac{x}{\|x\|},\frac{x}{\|x\|}\right\rangle\right)&=
  f\left(\frac{1}{ \|x\|^2}(\langle Ax,x\rangle -\langle By,y\rangle)+\frac{\|y\|^2}{ \|x\|^2}\left\langle B\frac{y}{\|y\|},\frac{y}{\|y\|}\right\rangle\right)\nonumber\\
  &\leq \frac{1}{ \|x\|^2}f(\langle Ax,x\rangle -\langle By,y\rangle)+\frac{\|y\|^2}{ \|x\|^2}f\left(\left\langle B\frac{y}{\|y\|},\frac{y}{\|y\|}\right\rangle\right)\\
  &\ \ \ -\frac{1}{ \|x\|^2}
  f\left(\frac{\|y\|^2}{ \|x\|^2}\left|\langle Ax,x\rangle -\langle By,y\rangle-\frac{1}{\|y\|^2}\langle By,y\rangle\right|\right)\nonumber\\
  & \ \ \ -
\frac{\|y\|^2}{ \|x\|^2}
  f\left(\frac{1}{ \|x\|^2}\left|\langle Ax,x\rangle -\langle By,y\rangle-\frac{1}{\|y\|^2}\langle By,y\rangle\right|\right)\nonumber.
\end{align}
Since $f$ is superquadratic, it follows from \eqref{kian} that
\begin{align}\label{qq2}
  f\left(\left\langle B\frac{y}{\|y\|},\frac{y}{\|y\|}\right\rangle\right)\leq \frac{1}{\|y\|^2}\langle f(B)y,y\rangle - \frac{1}{\|y\|^2}\langle f\left(\left|B-\frac{1}{\|y\|^2}\langle  By,y\rangle\right|\right)y,y\rangle
\end{align}

Multiplying both sides of  \eqref{qq1}  by $\|x\|^2$ and using \eqref{qq2}  we get
\begin{align*}
 \|x\|^2f\left(\left\langle A \frac{x}{\|x\|},\frac{x}{\|x\|}\right\rangle\right)&\leq f(\langle Ax,x\rangle -\langle By,y\rangle)+\langle f(B)y,y\rangle-\langle f\left(\left|B-\frac{1}{\|y\|^2}\langle  By,y\rangle\right|\right)y,y\rangle\\
 &\ \ -f\left(\|y\|^2 \left|\frac{1}{ \|x\|^2}\langle Ax,x\rangle -\frac{1}{ \|y\|^2}\langle By,y\rangle \right|\right)\\
 & \ \ - \|y\|^2 f\left( \left|\frac{1}{ \|x\|^2}\langle Ax,x\rangle -\frac{1}{ \|y\|^2}\langle By,y\rangle \right|\right),
\end{align*}
which concludes the result.


\end{document}